\documentclass[11pt,reqno]{amsart}
\usepackage{enumerate}
\usepackage{hyperref}
\usepackage{amsmath}
\usepackage[dvips]{graphicx}
\usepackage{color}
\usepackage{setspace}  
\usepackage{amscd} 
\usepackage{amsthm}
\usepackage{amssymb}
\title{Non-split linear sharply 2-transitive groups}
\author{Yair Glasner and Dennis D. Gulko}
\date{\today}
\subjclass[2010]{Primary 20B22}%
\keywords{sharply $2$-transitive, linear groups.}%

\newtheorem{theorem}{Theorem}[section]
\newtheorem{lemma}[theorem]{Lemma}

\newtheorem{prop}[theorem]{Proposition}
\newtheorem{remark}[theorem]{Remark}
\newtheorem{defi}[theorem]{Definition}
\newtheorem{cor}[theorem]{Corollary}
\newtheorem{question}[theorem]{Question}

\newcommand{\Qed}{\nobreak \ifvmode \relax \else
      \ifdim\lastskip<1.5em \hskip-\lastskip
      \hskip1.5em plus0em minus0.5em \fi \nobreak
      \vrule height0.75em width0.5em depth0.25em\fi}

\newcommand{\nn}{{\mathbf{3}}}
\newcommand{\Z}{{\mathbf{Z}}}
\newcommand{\R}{{\mathbf{R}}}
\newcommand{\C}{{\mathbf{C}}}
\newcommand{\Q}{{\mathbf{Q}}}

\renewcommand{\P}{{\mathbf{P}}}
\renewcommand{\Pr}{\pi}
\newcommand{\F}{{\mathbf{F}}}

\newcommand{\cU}{{\mathcal{U}}}
\newcommand{\arrow}{\rightarrow}
\newcommand{\embedding}{\hookrightarrow}

\newcommand{\sgn}{{\operatorname{sgn}}}

\newcommand{\SL}{{\operatorname{SL}}}
\newcommand{\trivgp}{\langle e \rangle}

\newcommand{\norm}[1]{\left\Vert#1\right\Vert}

\newcommand{\Span}{{\operatorname{Span}}}

\newcommand{\GL}{{\operatorname{GL}}}
\newcommand{\PGL}{{\operatorname{PGL}}}
\newcommand{\Inv}{{\operatorname{Inv}}}
\newcommand{\Char}{{\operatorname{Char}}}

\newcommand{\diag}{{\operatorname{Diag}}}

\newcommand{\bs}{\backslash}

\newcommand{\ord}{{\operatorname{Ord}}}

\newcommand{\pchar}{\operatorname{p-char}}

\def\go{\omega}

\def\bE{\begin{enumerate}}
\def\eE{\end{enumerate}}

\newcommand{\lsr}[1]{\left[#1\right]}
\newcommand{\lr}[1]{\left(#1\right)}
\newcommand{\LR}[1]{\left\{#1\right\}}

\newcommand{\bit}{\begin{itemize}}
\newcommand{\eit}{\end{itemize}}

\def\ker{\operatorname{ker}}

\begin{document}
\maketitle
\begin{abstract}
We give examples of countable linear groups $\Gamma < \SL_{\nn}(\R)$, with no nontrivial normal abelian subgroups, that admit a faithful sharply $2$-transitive action on a set. Without the linearity assumption, such groups were recently constructed by Rips, Segev, and Tent in \cite{RST:s2t}. Our examples are of permutational characteristic $2$, in the sense that involutions do not fix a point in the $2$-transitive action. \end{abstract}
\section{Introduction}
A sharply $2$-transitive group is, by definition, a permutation group $X \curvearrowleft \Gamma$ which acts transitively and freely on ordered pairs of distinct points. Such a group is called {\it{split}} if it admits a nontrivial normal abelian subgroup. In this case the normal abelian subgroup acts freely and transitively on $X$. The following question, that attracted the attention of algebraists for many years, was recently answered negatively, by Rips, Segev and Tent in \cite{RST:s2t}
\begin{question} \label{q:main}
Is every sharply $2$-transitive group split?
\end{question}
In Theorem \ref{thm:s2t} we show that the answer remains negative even in the setting of countable linear groups. This came to us as somewhat of a surprise because prior results of ours \cite{GlGu} and of Glauberman, Mann and Segev  \cite{GMS:s2t} ruled out the possibility of linear counterexamples, unless the permutational characteristic of the sharply $2$-transitive group is $2$ (see definition \ref{def:pchar}). 

Splitting implies a tame, algebraic, structure theory. With every split sharply $2$-transitive group $\Gamma$ one can associate a near field $N$. By definition a {\it{near field}} is just like a division ring, except that it is required to be distributive only from the right. The group $\Gamma$ then indeed {\it{splits}} as a semidirect product $\Gamma= N^* \ltimes N$ of the multiplicative and additive groups of the near field. Moreover the given sharply $2$-transitive action is the unique faithful primitive permutation representation of this group and it is equivalent to the natural action $N \curvearrowleft N^* \ltimes N$ by affine transformation $x^{(a,b)} = x\cdot a+b$. In this paper all permutation actions will be right actions, this also dictates our choice of right near fields. 

When $X \curvearrowleft \Gamma$ is sharply $2$-transitive there exists an element flipping any two points in $X$, whose square must be trivial. This gives rise to a large set of involutions $\Inv(\Gamma) \subset \Gamma$. Every such involution has to flip (many) pairs of points, and a pair of points determines the involution. Since $\Gamma$ is transitive on pairs of points all involutions are conjugate. Any nontrivial element, and in particular any involution, can have either $0$ or $1$ fixed points. If every involution does fix a point then the map $\Inv(\Gamma) \arrow X$ taking an involution to its fixed point is a $\Gamma$-invariant bijection. In this case it follows that the $\Gamma$-action on $\Inv(\Gamma)$ by conjugation is $2$-transitive. In particular, the order of the product of two distinct involutions is independent of the choice of the specific involutions (see \cite[Chapter 2]{Kerby:mst}, for more details). This gives rise to the following definition:
\begin{defi} \label{def:pchar}
Let $\Gamma$ be a sharply $2$-transitive permutation group on $X$. If the stabilizer of a point contains an involution let $p = \ord(\sigma \tau)$ be the order of the product of two distinct involutions. We define the {\it{permutational characteristic of $\Gamma$}} to be
$$\pchar(\Gamma)=\left\{ \begin{array}{ c l }
	2\hspace{8pt} & {\text{Involutions do not fix a point}}\\
	p \hspace{8pt} & p < \infty \\
	0\hspace{8pt} & p =\infty
\end{array}\right.$$
\end{defi}
\noindent It is not difficult to verify that $\pchar(\Gamma)$ is either $0$ or prime, and that it coincides with the characteristic of the near field whenever $\Gamma$ splits. We refer the readers to \cite{Kerby:mst}, \cite{GlGu} for more details. 

In two papers \cite{Z1:finite_s2t,Z2:finite_s2t} from 1936 H. Zassenhaus completed a full classification of  finite sharply $2$-transitive groups. He started by showing that every such group splits, and then gave a complete classification of finite near fields. Contrary to the situation for skew fields, this classification does involve nontrivial examples of finite near fields. In \cite{T1:cont_s2t,T2:cont_s2t} Tits proved that every  locally compact connected sharply $2$-transitive group splits. The near fields here are just $\R,\C$ or $\mathbf{H}$. In \cite{Kerby:mst}, \cite{Turk:s2t_char3} it was proved that every sharply $2$-transitive group $\Gamma$ with $\pchar(\Gamma)=3$ splits. In \cite{GlGu,GMS:s2t} it was shown that every linear sharply $2$-transitive group $\Gamma < \GL_n(k)$ with $\pchar(\Gamma) \ne 2$ splits. The first paper, by the current authors relied on an extra assumption that $\operatorname{char}(k) \ne 2$, this restriction was removed in the second paper, by Glauberman, Mann and Segev who also relaxed the linearity assumption. 

\begin{defi} \label{def:embedding_of_actions}
Let $A \bs G \stackrel{\phi}{\curvearrowleft} G, B \bs H \stackrel{\psi}{\curvearrowleft} H$ be transitive actions. We say that {\it{$\phi$ is embedded in $\psi$}} if $G \leq H$ is a subgroup and $A = B \cap G$. Equivalently, the map $A \bs G \arrow B \bs H$ given by $Ag \mapsto Bg$ is well defined, injective and $G$-equivariant. 
\end{defi}

Recently the first examples of nonsplit sharply $2$-transitive groups were given by Rips-Segev-Tent in \cite{RST:s2t}. In fact they show that any transitive group action $A \bs G \curvearrowleft G$ can be embedded into a nonsplit sharply $2$-transitive action $B \bs H \curvearrowleft H$ of permutational characteristic $2$; unless one of the following obvious obstructions hold (i) there exists a nontrivial group element $g \in G$ fixing two distinct points, (ii) there exists an involution $\sigma \in G$ fixing a point in $X$. There is a third obstruction mentioned in \cite{RST:s2t}, that the action $A \bs G \curvearrowleft G$, is not already sharply $2$-transitive but this additional assumption is not really necessary. Indeed if this action happens to be sharply $2$-transitive we can first replace the pair of groups $A < G$ by the pair $A < G_1 = G*Z$. The action $A \bs G_1 \curvearrowleft G_1$ is clearly not even primitive, but it satisfies all the other conditions of the theorem and we can proceed from here.

Our main theorem comes to show that such nonsplit sharply $2$-transitive groups of permutational characteristic $2$ can be constructed even within the realm of linear groups.
\begin{theorem} \label{thm:s2t} (Linear s-2-t groups)
Let $n \ge \nn$ and $L < \SL_n(\R)$ be a countable group that contains neither involutions nor nontrivial scalar matrices. Assume that $A \bs L \curvearrowleft L$ is a transitive action, such that the only element fixing more than a single point is the identity. Then there exists a larger countable group $L < H < \SL_n(\R)$ and an embedding of the given action into a nonsplit sharply $2$-transitive action $B \bs H \curvearrowleft H$. 

Moreover the group $H$ admits uncountably many nonequivalent faithful primitive permutation representations. 
\end{theorem}

The last statement of the theorem should be contrasted with the well known fact (see for example \cite[Appendix A]{GG:Primitive}) that a split sharply $2$-transitive group admits a unique (up to equivalence of  permutations representations) faithful primitive permutation representation. Thus the nonsplit examples of \cite{RST:s2t} are the first natural candidates for sharply $2$-transitive groups that admit multiple faithful primitive actions. In order to actually construct such examples linearity comes in handy. We appeal to the results of \cite{GG:AOS} which ensure that any Zariski dense countable subgroup of $\SL_n(\R)$ with trivial centre admits uncountably many nonequivalent faithful primitive actions. But linearity is not essential here. Using the methods of \cite{RST:s2t}, it is possible to construct groups that admit many faithful nonequivalent sharply $2$-transitive actions! This will be shown in a subsequent paper. 

In a previous version of this paper we were able to embed nonsplit sharply two transitive groups inside $SL_n(\R)$ only for $n \ge 6$. By the encouragement of the referee we improved the proof to allow for any $n \ge 3$. In a sense, made explicit by the following theorem, this is the best possible value. 
\begin{theorem} \label{thm:n2}
Let $k$ be a field and $\Gamma < \SL_2(k)$ a sharply $2$-transitive group. Then $\Gamma$ splits. 
\end{theorem}
\begin{proof}
We may assume that $\Char(k) = 2$, otherwise $\SL_2(k)$ does not contain sharply 2-transitive groups as it has only one involution. Let $t \in \Gamma$ be an involution. Replacing if necessary $k$ by its algebraic closure and conjugating we may assume that $t = \begin{pmatrix} 1 & 1 \\ 0 & 1 \end{pmatrix}$ is already in Jordan form. Thus $Z_{\Gamma}(t) < Z_{\SL_2(k)}(t) = \left\{ \left. \begin{pmatrix} 1 & \alpha \\ 0 & 1 \end{pmatrix} \right | \alpha \in k \right \}$ is contained in $\Inv^{*}(\Gamma) := \Inv(\Gamma) \cup \{id\}$. Since in a sharply 2-transitive group both $Z_{\Gamma}(t)$ and $\Inv^{*}(\Gamma)$ are simply transitive, we deduce that these two sets must be equal. So $\Inv^{*}(\Gamma)$ is an abelian normal subgroup of $\Gamma$ and the group splits.  
\end{proof}
The way this proof rules out fields of characteristic different from two is somewhat of a cheat. It would be more interesting to prove a similar statement for subgroups of groups such as $\GL_2(k)$ or $\PGL_2(k)$ that contain nontrivial involutions.

The proof of Theorem (\ref{thm:s2t}) appeals to the construction made in \cite{RST:s2t}. We show that under the conditions stated in the theorem, the same construction can be realized within the ambient group $\SL_n(\R)$. For this we rely on the structure (see Section \ref{sec:gen}) of $\SL_n(\R)$ as a locally compact, second countable, topological group. We use the dynamics of the action $\P^{n-1} \curvearrowleft \SL_n(\R)$ on the projective space. We also apply the Baire category theorem to the group $\SL_n(\R)$ itself and to some of its closed subgroups. Though we didn't verify all the details we are convinced that our method would still work upon replacing $\R$ by other local fields, namely finite extensions of $\Q_p$ and $\F_p((t))$\footnote{When the characteristic of the local field is $2$, (i.e. for $\F_{q}((t))$ with $q = 2^m$) this is significantly more difficult.}. The following question the question of smaller fields or rings where the above mentioned topological methods are no longer available. 
\begin{question}
Does the group $\SL_n(R)$ contain nonsplit sharply $2$-transitive subgroups for smaller rings such as the field of algebraic numbers $R=\overline{\Q}$, the rationals $R = \Q$ or even the integers $R = \Z$? 
\end{question}

We are grateful to the anonymous referee for numerous comments and in particular for providing Lemma \ref{lem:diagonal} and the concluding arguments of Theorem \ref{thm:s2t}. We find that they are more elegant than the original ones we had. The research of both authors was partially funded by Israel Science Foundation grant ISF 2095/15. 

\section{Proof}
\subsection{The main induction}
Given an involution $t \in \SL_n(\R)$ let $\R^n = W^{+}(t) \oplus W^{-}(t)$ denote the decomposition of $\R^n$ into its $\pm 1$ eigenspaces and $\Pr_t = \frac{1 - t}{2} :\R^n \arrow W^{-}(t)$ the projection on $W^{-}(t)$ along $W^{+}(t)$. When possible, we will omit $t$ from the notation writing $W^{\pm}$ instead of $W^{\pm}(t)$ and $\pi$ instead of $\pi_t$. Recall the following definition:
\begin{defi} 
 A subgroup $A\leq G$ is called \rm{malnormal} if for all $g\in G, g \not \in A$ we have $A^g\cap A=\trivgp$
\end{defi}
\noindent Our main technical theorem is the following linear version of \cite[Theorem 1.1]{RST:s2t}. Note that it follows from the assumptions of the theorem that $n \ge 3$. 
\begin{theorem} \label{thm:tec}
Let $G < \SL_n(\R)$ be a countable group, $-1 \ne t \in G$ an involution such that $\dim(W^{-}(t)) \ge 2$ and $A < G$ a malnormal subgroup containing no involutions. Assume further that: 
\begin{enumerate}
\item \label{itm:conj} all the involutions in $G$ are conjugate (in $G$).
\item \label{itm:es} If $W^{-}(t)$ is contained in an eigenspace of $g \Pr_t$ for some $g \in G$ then either $g =1$ or $g =t$. 
\end{enumerate}
Then for any element $v \in G, v \not \in A$ there exists a countable extension $G \le G_1 < \SL_n(\R)$ with a malnormal subgroup $A_1 < G_1$ containing no involutions, such that $A_1 \cap G = A$ and an element $f \in A_1$ such that $A_1 t f = A_1 v$. Moreover $G_1,A_1$ still satisfy the conditions (\ref{itm:conj}), (\ref{itm:es}) above.
\end{theorem}
The improvement of this theorem over its counterpart, Theorem 1.1 in \cite{RST:s2t} is the linearity: when $G$ is linear, its extension $G_1$ is also linear and within the same ambient matrix group $\SL_n(\R)$. The eventual existence of a nonsplit sharply 2-transitive extension $G < H < \SL_n(\R)$ gives rise to two unavoidable restrictions which were not necessary in \cite{RST:s2t}. Firstly all the involutions in $G$  are eventually conjugate in $H$, therefore they must be represented by conjugate matrices in $\SL_n(\R)$. Secondly $G$ cannot contain nontrivial scalar matrices as these will give rise to a nontrivial centre for $H$ which is absurd. These restrictions have to do with the way $G$ and its extensions are embedded in $\SL_n(\R)$. They cannot be stated in purely group theoretic terms, which is the reason why they appear here and not in \cite{RST:s2t}. The conditions we impose (\ref{itm:conj}) and (\ref{itm:es}) are somewhat stronger than these obvious restrictions but they are of the same spirit. 

It is important that the resulting group action $A_1 \bs G_1 \curvearrowleft G_1$ is again subject to the auxiliary conditions (\ref{itm:conj}), (\ref{itm:es}), and that $G_1$ it is contained in the same ambient matrix group. This enables us to use Theorem \ref{thm:tec} inductively - increasing the transitivity level with each application of the theorem.
\begin{cor} \label{cor:tec} 
Under the assumptions of Theorem \ref{thm:tec}, there exists a countable extension $G < H < \SL_n(\R)$ and a subgroup $B < H$ with the following properties:
\begin{itemize}
\item $B \cap G = A$,
\item $B$ is malnormal in $H$,
\item for every $v \in H, v \not \in B$ there exists an element $f \in B$ such that $Bv=Btf$,
\item $B$ does not contain involutions.
\end{itemize} 
In permutation theoretic terminology these items mean that the action $B \bs H \curvearrowleft H$, extends the given action, that it is free and transitive on ordered pairs of points, and that the resulting sharply $2$-transitive group is of permutational characteristic $2$. 
\end{cor}
\begin{proof}
Let us enumerate the elements of the group $G$ that are not in $A$ as follows $G \smallsetminus A = \{v_1,v_2, \ldots\}$. Applying Theorem \ref{thm:tec} inductively we obtain a sequence of extensions $(A,G) \leq (A_1,G_1) \leq (A_2,G_2) \leq \ldots$, such that $G_i \cap A_{i+1} = A_i$. These come together with elements $f_i \in A_i$ such that $A_i t f_i = A_i v_i$. Let $G_{\omega} = \cup_{i=1}^{\infty} G_i$ and $A_{\omega} = \cup_{i=1}^{\infty} A_i$. Note that $G_i \cap A_{\omega} = A_i, \ \forall i$. Otherwise for any $g \in G_i \cap A_{\omega}, g \not \in A_i$, after replacing $i$ by the maximal index for which this still holds, we obtain a contradiction to $G_{i} \cap A_{i+1}=A_{i}$. In particular $G \cap A_{\omega} = A$. By construction every $v \in G \smallsetminus A$ there exists $f \in A_{\omega}$ such that $A_{\omega}tf=A_{\omega}v$. We claim that the pair of groups $A_{\omega} < G_{\omega }$ satisfy again the conditions of Theorem  \ref{thm:tec}, with respect to the same involution $t$. Indeed any violation of these conditions can be expressed using finitely many elements, and hence is witnessed by $A_n < G_n$ for some finite $n$. For example assume $A_{\omega} < G_{\omega}$ fails to be malnormal then $e \ne a = ga'g^{-1} \in A \cap A^g$ for some $a,a' \in A_{\omega}, g \in G_{\omega}, \ g \not \in  A_{\omega}$. Now for some finite $n$ we have $a, a' \in A_n, g \in G_n, \ g \not \in A_n$. Since we already saw hat $g \not \in A_n$ we have a contradiction to the malnormality of $A_n$ in $G_n$, that was guaranteed by successive applications of Theorem \ref{thm:tec}. A similar argument establishes conditions (\ref{itm:conj}) and (\ref{itm:es}).

Applying the procedure of the previous paragraph to $(A_{\omega},G_{\omega})$ and continuing inductively we obtain a sequence $(A_{\omega}, G_{\omega}) \leq (A_{2\omega},G_{2\omega}), \leq \ldots$ The pair of groups $H = G_{\omega \cdot \omega} = \cup_{i=1}^{\infty} G_{i \omega}$, $B = A_{\omega \cdot \omega} = \cup_{i =1}^{\infty}A_{i\omega}$ satisfy all the desired properties. 
\end{proof}

\subsection{Generic constructions} \label{sec:gen}
A Baire space is a topological space satisfying the conclusion of Baire's category theorem. Namely a space in which a countable intersection of dense open sets is still dense. Examples include locally compact Hausdorff spaces as well as Polish spaces, i.e. separable spaces admitting a complete metric which is compatible with the topology. A subset in a Baire space is called {\it{residual}} or {\it{generic}} if it contains the intersection of countably many dense open sets. Residual sets in Baire spaces play the role of {\it{large sets}}, quite similar to the role played by sets of full measure in the measure theoretic setting. See for example \cite{ox:mc},\cite[Chapter 2.1]{Rudin:FA}. When we say that a property $P$ {\it{holds generically}} in a Baire space $X$ we mean that the set of all elements of $X$ satisfying $P$ is residual. The complement of a residual set is called a {\it{meagre set}}.

In this paper we will work with a few Baire spaces: The groups $\SL_n(\R)$ and $\SL_k(\R) \times \SL_{n-k}(\R)$, the projective space $\P= (\R^n\setminus\LR{0}) / \R^{*}$ and various products of these spaces. All of these admit a standard topology, coming for example from their structure as analytic manifolds. This topology is locally compact, metrizable and hence clearly Baire. In addition all these spaces are also real points of smooth, irreducible algebraic varieties (see for example \cite[section 2.1.4]{PR:Book} for the algebraic groups, \cite[Section I.2]{Hart:Book} for the projective plane). The interplay between the Hausdorff and the Zariski topology, gives rise to a large family of open dense sets. 
\begin{lemma} \label{lem:ZtHt} Every non-empty Zariski open subset $\emptyset \ne U < Z$ in  one of the above spaces is dense and open in the Hausdorff topology on $Z$. A countable intersection of such Zariski open sets is residual and in particular nonempty. 
\end{lemma}
\begin{proof} $U$ is automatically open in the finer Hausdorff topology. Since $Z$ is irreducible its complement is a subvariety of lower dimension which cannot contain an open set. See \cite[Section 3.1, Lemma 3.2]{PR:Book} for details. The last statement follows directly from Baire's theorem. \end{proof} 
\subsection{Projective dynamics}
All our notation is taken from \cite[Section 3]{BG:Dense_Free} and we refer the readers to that paper for more details. 

We use extensively the action $\P \curvearrowleft \SL_n(\R)$ where $\P= (\R^n\setminus\LR{0}) / \R^{*}$ is the projective space. If $0 \ne v \in \R^n, \langle 0 \rangle \ne W < \R^n$ are a nontrivial vector and subspace in $\R^n$ we denote the corresponding projective point, and subspace by $[v] \in \P, [W] < \P$. Fix a norm $\norm{\cdot}$ on $\R^n$, this gives rise to a norm on the exterior product $\R^n \wedge \R^n$ which is used  in turn to endow the projective space $\P$ with the metric 
$$d([v],[w]) = \frac{\norm{v \wedge w}}{\norm{v} \cdot \norm{w}}.$$
See the above reference for more details on this metric. We will denote the $\epsilon$ neighborhood of a set $\Omega$ in this metric by
$(\Omega)_{\epsilon} = \LR{x \in \P \ | \ d(x,\Omega) < \epsilon}$.

Any nontrivial matrix $0 \ne M \in M_n(\R)$ gives rise to a partially defined map $[M]: \P \smallsetminus [\ker(M)] \arrow \P$. A projective point $[v] \in \P$ is moved by $[M]$ if and only if $v$ is not an eigenvector of $M$. Indeed $[M]$ is not defined on $[v]$ if and only if $v$ is an eigenvector with eigenvalue $0$ and $[v]$ is a fixed point of $[M]$ if and only if $v$ is an eigenvector with a nonzero eigenvalue. With this terminology for example, condition (\ref{itm:es}) in Theorem \ref{thm:tec} requires that whenever $g \in G, \ g \not \in \{1,t\}$ then the map $[g \Pr_t]  = \lsr{\frac{g - gt}{2}}$ should move at least one projective point in $[W^{-}(t)]$. If $M \in \GL_n(\R)$ then the map $[M]$ is everywhere defined and gives rise to a homeomorphism of the whole projective space. 

\begin{lemma} \label{lem:general_pos}
Let $B \in \GL_r(\R)$. If $v_0,v_1,\ldots, v_r \in \R^r$ are $r+1$ vectors in general position that are all eigenvectors of $B$. then $B$ is a scalar matrix. 
\end{lemma} 
\begin{proof}
By definition, $r+1$ vectors in $\R^r$ are in {\it{general position}} if any $i$ of them span an $i$-dimensional subspace as long as $i \le r$. By counting considerations there must be at least one eigenspace $V< \R^r$ of dimension $l$ containing at least $l+1$ of the eigenvectors $\{v_0,v_1,\ldots, v_r\}$. Since the vectors are in general position necessarily $l=r$; so $B$ has an eigenvalue with $r$-linearly independent eigenvectors and is hence a scalar. 
\end{proof}

\subsection{Linearity of free products}
The following proposition asserts that the free product of two countable linear groups $G,H < \SL_n(\R)$ is still linear inside the same ambient matrix group. Similar linearity statements for free products are well known, see for example \cite{shalen}. Shalen's statement is stronger than ours as it is not restricted to countable groups. Our statement is tailored for its use in the main induction. In particular, in anticipation of (\ref{itm:es}) of Theorem \ref{thm:tec}, we require the embedding $G*H \hookrightarrow \SL_n(\R)$ to be injective {\it{relative to some fixed linear projection $\pi$}}. A requirement made explicit in the statement of the theorem. We provide a full proof. This also serves to demonstrate our method, that will  appear again in the proof of Theorem \ref{thm:HNN} on linearity of HNN extensions.  
\begin{prop} \label{prop:free_prod} Let $G,H < \SL_n(\R)$ be two countable groups that contain no nontrivial scalar matrices.  Let $\Pr: \R^n \arrow W$ be a linear projection with $2 \le \dim(W) < n$. For $f \in \SL_n(\R)$ let $\Phi_f:G*H \arrow \SL_n(\R)$ be defined by $\Phi_f(g)=g, \forall g \in G$ and $\Phi_f(h) = fhf^{-1}, \ \forall h \in H$. 

Then for a Baire generic choice of $f \in \SL_n(\R)$ the map $\Phi_f$ is injective and its image contains no nontrivial scalar matrices. Moreover whenever $\go \in G*H$ and $\go \not \in G$ then $W$ is not contained in an eigenspace of $\Phi_f(\go) \Pr$. 
\end{prop}
\begin{proof}
We define
\begin{eqnarray*}
\cU & = &  \bigcap_{\go \in (G*H), \ \go \not \in G} \cU(\go), \qquad {\text{where}} \\ 
\cU(\go) & = & \LR{f\in\SL_n(\R): \Phi_f(\go) \Pr {\text{ is not scalar on }} W}.
\end{eqnarray*}
If we show that, for every $\go \in G*H, \go \not  \in G$, that  $\cU(\go)$ is non empty Zariski open Lemma \ref{lem:ZtHt} will show that $\cU$ is residual and all the claims of the theorem will follow. 

We shall hence fix such an element $\go$ and write it as a reduced word $\go = g_1 h_1 \cdot...\cdot g_k h_k $ with $g_i, \in G, h_i \in H$, all of them nontrivial except possibly $g_1$ and $h_k$. Since $\go \not \in G$ we exclude the possibility $\go = g_1$. Assume that $\dim(W)=r$ and $\{w_0,w_1,\ldots, w_r \}$ are $r+1$ vectors in general position in $W$. By Lemma \ref{lem:general_pos} the complement of $\cU(\go)$ is characterized by the equations $\{w_i^{\Phi_f (\go) \Pr } \wedge  w_i = 0\}_{0 \le i \le r}$, so $\cU(\go)$ is Zariski open in $\SL_n(R)$. It remains to exhibit $f \in \SL_n(\R)$ such that $W$ is not contained in an eigenspace of $\Phi_f(\go)\Pr$, thus proving that $\cU(\go)$ is nonempty. 

We now collect the data that will be used to define the desired element $f$. Choose a basis $\{v_1,v_2,\ldots, v_n\}$ for $\R^n$, a projective point  $x \in [W]$ and a number $L \gg 0$. In the projective space $\P= (\R^n\setminus\LR{0}) / \R^{*}$ let $a^{+} = [v_1]$, $a^{-} = [v_2]$, $H^{+} = [\Span \{v_2,v_3,v_4,\ldots, v_n\}]$ and $H^{-} = [\Span\{v_1,v_3,v_4,\ldots, v_n\}]$ be two projective points and two projective hyperplanes corresponding to these vectors.  We use all of this data to define $f = f(L) \in \SL_n(\R)$ by specifying its action on this basis:
$$v_1^f = L v_1, \qquad v_2^f = \frac{1}{L} v_2, \qquad v_i^f = v_i , \ \ \ 3 \le i \le n.$$
The dynamics of $f$ on $\P$ is proximal in the sense that 
\begin{equation} \label{eqn:prx}
\lim_{L \arrow \infty} y^{f(L)^{\pm 1}} = a^{\pm}, \forall y \not \in H^{\pm}, \qquad {\text{respectively}}.
\end{equation}

Note that
$$x^{\Phi_f(\go) \Pr} = x^{g_1 f h_1 f^{-1} g_2 f h_2 f^{-1} \ldots g_k f h_k f^{-1} \Pr}$$
Consider the points $x_1 = x^{g_1 f}$, $y_1 = x_1^{h_1 f^{-1}}$, $x_2 = y_1^{g_2 f}$, up to  $y_k= x_{k}^{h_k f^{-1}} = x^{\Phi_f(\go)}$. We will also write $x_i(L),y_i(L)$ to emphasize the dependence of these points on $L$. Our strategy is to choose all of the above data in such a way that these points alternate, $x_i$ coming arbitrarily close to $a^{+}$ and $y_i$ arbitrarily close to $a^{-}$, when $L \arrow \infty$. If we insist also that $x \ne (a^{-})^{\Pr}$ (or that $x \ne (a^{-})^{g_k \Pr}$ when $h_k = 1$) it will follow that when $L$ is large enough $\Phi_f(\go) \Pr$ does not fix the projective point $x \in [W]$ and the claim will follow. To carry this out we require the following:
\begin{enumerate}[(a)]
\item \label{itm:c1}  $x^{g_1} \not \in H^{+}$.
\item  \label{itm:c2} $(a^{-})^{g_i} \not \in H^{+}, {\text{ for }} 1< i \le k$.
\item \label{itm:c3} $(a^{+})^{h_i} \not \in H^{-}, {\text{ for }} 1 \le i < k.$ and also for $i = k$ if $h_k \ne 1$. 
\item \label{itm:c4} $x \ne (a^{-})^{\Pr}$ if $h_k \ne 1$, or $x \ne (a^{-})^{g_k \Pr}$ when $h_k = 1$.
\end{enumerate} 
It is an exercise in linear algebra to prove the existence of a collection of vectors satisfying all these conditions. Since we are already in the Baire category business, let us argue that it is enough to verify the satisfiability of each condition separately. What we are looking for is an ordered basis $\{v_1,\ldots, v_n\}$ for $\R^n$ together with projective point $x \in [W]$. We take the variety $C=\SL_n(\R) \times [W]$ as our parameter space, where the ordered basis is given by the columns of the matrix. Since we restricted to matrices in $\SL_n(\R)$, $C$ is an irreducible variety as a product of a connected algebraic group and a projective space. Clearly the conditions (\ref{itm:c1})-(\ref{itm:c4}) define Zariski open subsets of $C$ and by Lemma \ref{lem:ZtHt}, it is enough to show that each one of these separately is nonvoid. To see that condition (\ref{itm:c1}) is nonvoid we choose our basis in such a way that $([W])^{g_1} \not \subset H^{+}$ and then choose $x \in [W]$ to be any point such that $x^{g_1} \not \in H^{+}$. Since $\{g_2,\ldots, g_{k}\}$ are by assumption non scalar matrices we can choose $v_2$ which is not an eigenvector of any of these and then complete it into a basis in such a way that Condition (\ref{itm:c2}) is satisfied. Condition (\ref{itm:c3}) is addressed similarly. Finally by assumption $\dim([W]) = \dim(W)-1 \ge 1$, so that after we have chosen the basis $\{v_1,\ldots, v_n\}$ we can always choose $x \in [W]$ in such a way that Condition (\ref{itm:c4}) is satisfied.

We now argue, by induction on $i$ that $\lim_{L \arrow \infty} x_i(L) = a^{+}$ and that $\lim_{L \arrow \infty} y_i(L) = a^{-}$ for all $1 \le i < k$, and also for $i=k$ in the case where $h_k \ne 1$. Condition (\ref{itm:c1}) guarantees that $x^{g_1} \not \in H^{+}$, hence $\lim_{L \arrow \infty} x_1(L) = \lim_{L \arrow \infty} (x^{g_1})^{f(L)} = a^{+}$ by Equation (\ref{eqn:prx}). Separating into even and odd steps of the induction we obtain:
\begin{eqnarray*}
\lim_{L \arrow \infty} y_{i}(L) & = & \lim_{L \arrow \infty} (x_i(L))^ {h_i f^{-1}} = \lim_{L \arrow \infty} (a^{+})^{h_i f^{-1}} = a^{-} \\
\lim_{L \arrow \infty} x_{i+1}(L) & = & \lim_{L \arrow \infty} (y_i(L))^ {g_{i+1} f} = \lim_{L \arrow \infty} (a^{-})^{g_{i+1} f} = a^{+}
\end{eqnarray*}
Where, in each line, the first equality comes from the definition of $y_i(L), x_i(L)$, the second equality uses the induction hypothesis and the last follows from Equation (\ref{eqn:prx}) in view of conditions (\ref{itm:c3})  or (\ref{itm:c2}) respectively. This induction shows that 
$$\lim_{L \arrow \infty} x^{\Phi_f(w )} = 
   \left \{
   \begin{array}{ll}
   a^{-} & {\text{if $h_1 \ne 1$}}  \\
   \left(a^{-}\right)^{g_k} & {\text{if $h_1=1$}}  
   \end{array} \right. ,
$$
now (\ref{itm:c4}) is exactly what is needed to conclude the proof. 

 \end{proof}

\subsection{Proof of the main theorem} \label{sec:linearity}
In this section we prove Theorem \ref{thm:tec}. Throughout we assume that $G,A,t,v$ are as given in that theorem. The following proposition reduces the proof to two special cases. 
 \begin{prop} \label{prop:hyp}
It is enough to prove Theorem \ref{thm:tec} under the additional assumption that $v \not \in AtA$ and that either $v^{-1} \not \in AvA$ or $v$ is an involution.
\end{prop} 
\begin{proof} If $v = ata' \in AtA$ then one can take $A_1=A, G_1=G, f=a'$ and conclude. \cite{RST:s2t} mention also the condition $v^{-1} \not \in AtA$, but this is redundant because as $t$ is an involution it is equivalent to $v \not \in AtA$. Other than that reduction is identical to that carried out in \cite[Section 2]{RST:s2t}.  
\end{proof}
\noindent The next two theorems address the two cases singled out in the above proposition respectively. 
\begin{theorem}(Free product) \label{thm:free_prod} Let $A,G,t,v$ be as in Theorem \ref{thm:tec} and assume that $v^{-1} \not \in AvA$. Then there exists an element $\ell
\in \SL_n(\R)$ of infinite order such that the subgroup $G_1 = \langle G,\ell \rangle$ is isomorphic to the free product $G * \langle \ell \rangle$. Furthermore if we set $A_1 := \langle A, \ell, t \ell v^{-1} \rangle,$ then:
\begin{enumerate}
\item \label{itm:f1} $A_1$ contains no involutions,
\item \label{itm:f2} $A_1$ is malnormal in $G_1$, 
\item \label{itm:f3} $A_1t \ell = A_1 v$, 
\item \label{itm:f4} $G \cap A_1 = A$,
\item \label{itm:f5} all the involutions in $G_1$ are conjugate (in $G_1$),
\item \label{itm:f6} If $W^{-}(t)$ is contained in an eigenspace of $g \Pr_t$ for $g \in G_1$ then either $g =1$ or $g =t$.
\end{enumerate}
\end{theorem}
\begin{proof}
The existence of an element $\ell$ giving rise to an isomorphism $\langle G,\ell \rangle \cong G_1$ and  satisfying (\ref{itm:f6}) follows upon application of Proposition \ref{prop:free_prod} to the given group $G$ along with any infinite cyclic subgroup $H < \SL_n(\R)$ and the projection  $\pi = \pi_t$.  

The rest of the desired properties are just claims about the abstract group $G_1$. Properties (\ref{itm:f1}),(\ref{itm:f2}),(\ref{itm:f3}) and (\ref{itm:f4}) were established in \cite[Theorem 3.1]{RST:s2t}. The condition $v^{-1} \not \in AvA$ was crucial in that proof, but we will not use it again here. Finally let $\sigma \in G*\Z$ be any involution. Being an element of finite order it must stabilize a vertex in the Bass-Serre tree corresponding to the free product $G*\Z$. Since $\Z$ contains no involutions this means that $\sigma$ is conjugate into $G$. But by assumption all involutions in $G$ are already conjugate. This establishes (\ref{itm:f5}) and concludes the proof of the theorem.
\end{proof}
\begin{theorem}(HNN extension) \label{thm:HNN}
Let $A,G,t,v$ be as in Theorem \ref{thm:tec} and assume that $v$ is an involution. Then there exists an element $\ell \in \SL_n(\R)$ of infinite order such that the group $\langle G, \ell\rangle$ is isomorphic to the HNN extension product $G_1 := \langle G, \ell \ | \ \ell^{-1}t \ell=v \rangle$. Furthermore if we set $A_1:=\langle A, \ell \rangle$ then all the conclusions (\ref{itm:f1})-(\ref{itm:f6}) of the previous theorem hold.
\end{theorem}
\begin{proof}
Note first that properties (\ref{itm:f1}),(\ref{itm:f2}),(\ref{itm:f3}), (\ref{itm:f4}) and (\ref{itm:f5}) have nothing to do with the linear realization, they are purely algebraic statements about the groups $A_1 < G_1$. The first four were established in \cite[Theorem 4.1]{RST:s2t}. To show (\ref{itm:f5}) we argue on the Bass-Serre tree of the HNN extension $G_1 = \langle G, \ell \ | \ \ell^{-1}t \ell = v \rangle$. Any involution $\sigma \in G_1$, being an element of finite order, fixes a point in (the geometric realizatdion of) the tree. Since the action on the tree does not invert edges, $\sigma$ fixes a vertex and is thus conjugate into $G$ where, by assumption, all involutions are already conjugate.

In view of the assumption that all involutions in $G$ are already conjugate, we claim that the two HNN extensions $G_1=\langle G,\ell \ | \ \ell^{-1}t \ell=v\rangle$ and $G_2=\langle G,k \ | \ k^{-1}tk=t\rangle$ are isomorphic. Fix an element $h \in G$ such that $h^{-1}th = v$. Let $F: G_1 \arrow G_2$ be defined by the requirement that $F(g) = g, \forall g \in G, F(\ell) = k h$ and $I: G_2 \arrow G_1$ be  defined by $I(g)=g, \forall g \in G, I(k) = \ell h^{-1}$. Since $I \circ F$ fixes pointwise both $G$ and $\ell$ it must be the identity of $G_1$ and similarly $F \circ I$ is the identity of $G_2$. Now, in view of this isomorphism it would be enough to exhibit an element $k \in \SL_n(\R)$ such that the group $\langle G,k \rangle$ is isomorphic to the group  the group $\langle G,k \ | \ k^{-1}tk=t\rangle$ and such that Property (\ref{itm:f6}) is satisfied. The element $\ell$ in this case will be given by $\ell = kh$. 

Conjugating $G$ we may assume that $t = \diag\lr{-1,...,-1,1,...,1}$ is already a diagonal matrix.  Thus, if $e_1,...,e_n$ denote the vectors of the standard basis then $W^{-}:=W^{-}(t) = \langle e_1, e_2, \ldots,e_r \rangle,$ and $W^{+}:=W^{+}(t) =\langle e_{r+1},e_{r+2}, \ldots, e_n \rangle$ are, respectively, the $\mp 1$ eigenspaces of $t$. By assumption $r \ge 2$. 

Let $Z = \SL(W^{-}) \times \SL(W^{+}) \leq Z_{\SL_n(\R)}(t)$. This group is isomorphic to $\SL_{r}(\R) \times \SL_{n-r}(\R)$ and in particular it is a connected (and hence irreducible) closed subgroup of $\SL_n(\R)$. For any $u \in Z$ we obtain a homomorphism. $\Psi_u : G_2 \arrow \SL_n(\R)$ given by $g \mapsto g, \forall g \in G$ and $k \mapsto u$. 

Our goal is to find $u \in Z$ such that for every $w \in G_2, \ w \not \in \{1,t\}$ the matrix $\Psi_u(w) \Pr_t$ does not fix $[W^{-}]$ pointwise. This will show that $\Psi_u$ is an isomorphism onto its image, and establish (\ref{itm:f6}) at the same time. Thus concluding the proof of the theorem. 

A Baire category argument allows us to deal with each word $w \in G_1, \ w \not \in \{1,t\}$ separately. Let $\{w_0,w_1,\ldots, w_r\} \subset W^{-}$ be $r+1$ vectors in general position. By  Lemma \ref{lem:general_pos}
\begin{eqnarray*}
\cU_{w} & := & \left\{u \in Z \ | \  W^{-} {\text{ is not contained in an eigenspace of }} \Psi_u(w) \Pr_t \right\} \\
& = & Z \smallsetminus \left(\bigcap_{i=0}^r \left\{u \in Z \ | \ w_i ^{\Psi_u(w) \Pr_t} \wedge  w_i = 0 \right\} \right)
\end{eqnarray*} 
is the complement of a closed subvariety of $Z$. By Lemma \ref{lem:ZtHt} it is enough to show that each one of these sets is nonempty to deduce that their intersection $$\cU = \bigcap_{w \in G_1 \smallsetminus \{1,t\}} \cU_{w},$$ is a dense $G_{\delta}$ subset of $Z$. 

From here on fix $w \in G_1 \smallsetminus \{1,t\}$ and write this element in a reduced canonical form as $w=g_1 k^{\delta_1} g_2 \ldots k^{\delta_m}g_{m+1}$ where $g_i \in G, \delta_i \in \Z$. Let us also denote 
$$s_i = \sgn(\delta_i) = \left\{ \begin{array}{ll} + & {\text{ if }} \delta_i > 0 \\ - & {\text{ if }} \delta_i < 0 \end{array} \right.$$
Being reduced means that $g_i \ne 1$, for $i \in \{2,\ldots,m\}$, that $\delta_i \ne 0, \ \forall i$ and that whenever $g_{i+1}=t$ then $s_i = s_{i+1}$. In particular it is not important how $s_i$ is defined when $\delta_i=0$. We will conclude the proof once we find some $u \in Z$ and a projective point $x \in [W^{-}]$ such that 
\begin{equation} \label{eqn:goal} 
x^{\Psi_u(w) \Pr_t} = x^{g_1 u^{\delta_1} g_2 \ldots u^{\delta_m}g_{m+1} \Pr_t} \ne x.
\end{equation} 

We define our element $u$ using the the following data: (i) A basis $\{v_1,v_2, \ldots, v_n\}$ for $\R^n$, (ii) a number $L > 0$ and (iii) a point $x \in [W^{-}]$. We always require the basis to be compatible with the direct sum decomposition $\R^n = W^{+} \oplus W^{-}$ in the sense that $W^{-} = \Span \{v_1,\ldots, v_r\}$ and $W^{+} = \Span \{v_{r+1}, \ldots, v_n\}$. With our basis we associate two projective points and hyperplanes as follows: $a^{+} = [v_1], a^{-} = [v_2] \in \P$, $H^{+} = [\Span \{v_2,v_3,\ldots, v_n\}]$, $H^{-} = [\Span\{v_1,v_3,\ldots, v_n\}] < \P$. By our assumption $r\ge2$ so that $x,a^{\pm} \in [W^{-}]$.  Using all these we define $u = u(L) \in Z$ by:
$$v_1^{u} = L v_1, \qquad v_2^{u} = \frac{1}{L} v_2, \qquad v_i^{u} = v_i , \ \ \ \forall 3 \le i \le n.$$
The dynamics of the element $u$ on the projective plane $\P$ is very proximal in the sense that 
\begin{equation} \label{eqn:prox}
\lim_{L \arrow \infty} y^{u(L)^{\pm 1}} = a^{\pm}, \qquad \forall y \not \in H^{\pm} {\text{ respectively}}.
\end{equation}

Let $S = \{g_1,g_1^{-1}, g_2,g_2^{-1}, \ldots, g_{m+1}^{-1}\}$ and set $S_0 = S \smallsetminus \{1,t\}$. With Equation (\ref{eqn:goal}) in mind we impose the following requirements on the datum used to define $u$.
\begin{enumerate}[(a)]
\item \label{itm:x1} $(a^{\delta})^{h} \not \in H^{\eta}, \qquad \forall h \in S_0, \delta, \eta \in \{\pm\}.$
\item \label{itm:x2} $x^{h \Pr_t} \ne x, \ \forall h \in S_0$
\item \label{itm:x3} $x^{g_{1}} \not \in H^{+} \cup H^{-}$
\item \label{itm:x4} $(a^{\pm})^{g_{m+1} \Pr_t} \ne x$
\end{enumerate}

We first argue that it is possible to satisfy all the conditions (\ref{itm:x1})--(\ref{itm:x4}). We are looking for a basis compatible with the direct sum decomposition $W^{+} \oplus W^{-}$ together with one additional projective point. Within the variety $\GL_r(\R) \times \GL_{n-r}(\R) \times \P$ of possible choices, we seek a solution inside the {\it{connected}} subvariety $C = \SL_r(\R) \times \SL_{n-r}(\R) \times \P$. Since all the requirements (\ref{itm:x1})-(\ref{itm:x4}) are Zariski open it is enough to show that each one separately is nonvoid to deduce that they are simultaneously realized on a dense open subset of $C$. 

The $4|S_0|$ requirements expressed in (\ref{itm:x1}) come in two different flavors: $(a^{\pm})^h \not \in H^{\pm}$ and $(a^{\pm})^h \not \in H^{\mp}$. The first  type is trivial and left to the reader. As an example for the second one consider $(a^{+})^h \not \in H^{-}$. By the definition of $a^{+},H^{-}$ this is equivalent to $v_1^{h \Pr_t} \not \in \Span\{v_1,v_3,\ldots, v_r\}$ which is satisfied by $(\{v_1, v_2 = v_1^{h \Pr_t}, v_3, \ldots, v_n\},x) \in C$ whenever $v_1 \in W^{-}$ is not an eigenvector of $h \Pr_t$. There are  such vectors by the assumption that $W^{-}$ is not contained in an eigenspace of $h \Pr_t$ for $h \in S_0$. Similarly the conditions $x^{h \Pr_t} \ne x$, appearing in (\ref{itm:x2}) are satisfied by choosing $x = [v]$ where $v \in W^{-}$ is not an eigenvector of $h \Pr_t$.  If $g_1 \in \{1,t\}$ then (\ref{itm:x3}) is equivalent to $x \not \in H^{-} \cup H^{+}$, which is possible because $W^{-} \not \subset H^{+} \cup H^{-}$. If $g_1 \in S_0$ we appeal to (\ref{itm:x1}) showing that (\ref{itm:x3}) is satisfied by the choice $x= a^{+}$. Finally as $\dim([W^{-}]) \ge 1$ this set is infinite, so one can choose $x \in [W^{-}]$ avoiding the two points $(a^{\pm})^{g_{m+1} \Pr_t}$, thus satisfying (\ref{itm:x4}). 

Having constructed $u$ we proceed to prove Equation (\ref{eqn:goal}). If $m = 0$ then $g = g_1 = g_{m+1} \not \in \{1,t\}$  and  Equation (\ref{eqn:goal}) follows directly form (\ref{itm:x2}). Otherwise we define recursively a sequence of projective points $x_0 = x$ and $x_{i+1}= x_{i+1}(L) = x_i^{g_{i+1}u^{\delta_{i+1}}}$ and argue by induction on $j$ that 
$$\lim_{L \arrow \infty} x_j(L) = a^{s_j}, \quad \forall 1 \le j \le m.$$ 
For $j = 1$ we are looking to prove that
$\lim_{L \arrow \infty} x_1 = \lim_{L \arrow \infty} \left(x^{g_1}\right)^{u^{\delta_1}} = a^{s_1},$ which is true by Equation (\ref{eqn:prox}) in view of (\ref{itm:x3}). For $j+1$ the induction hypothesis yields
$$\lim_{L \arrow \infty} x_{j+1} = \lim_{L \arrow \infty} (x_j)^{g_{j+1} u^{\delta_{j+1}}} = \lim_{L \arrow \infty} \left( (a^{s_j})^{g_{j+1}} \right)^{u^{\delta_{j+1}}}.$$
If $g_{j+1} = t$ then $s_{j+1} = s_{j}$, because $w$ is a reduced word so $(a^{s_j})^{g_{j+1}} = (a^{s_j})^{t} = a^{s_j} \not \in H^{s_j} = H^{s_{j+1}}$. Otherwise $g_{j+1} \in S_0$ and $(a^{s_j})^{g_{j+1}} \not \in H^{s_{j+1}} $ by (\ref{itm:x1}). In both cases we obtain $\lim_{L \arrow \infty} x_{j+1} = a^{s_{j+1}}$ using Equation (\ref{eqn:prox}). This completes the induction. Our goal, namely Equation (\ref{eqn:goal}), now follows as 
$$\lim_{L \arrow \infty} x^{\Psi_u(w) \Pr_t} = (a^{s_m})^{g_{m+1} \Pr_t} \ne x,$$
where the last inequality is (\ref{itm:x4}). The existence, and in fact genericity, of elements $u \in \SL_r(\R) \times \SL_{n-r}(\R)$ for which $\Psi_u$ is injective and satisfies requirement (\ref{itm:f6}) follows. This concludes the proof of the theorem. 
\end{proof}
\begin{proof}[Proof of Theorem \ref{thm:tec}]
In view of Proposition \ref{prop:hyp}, Theorem \ref{thm:tec} now follows from Theorem \ref{thm:free_prod} and Theorem \ref{thm:HNN}. Corollary \ref{cor:tec} also directly follows. 
\end{proof}
\noindent Before turning to the proof of Theorem \ref{thm:s2t} we require the following.
\begin{lemma} \label{lem:diagonal}
Let $X \curvearrowleft G$ be a faithful sharply $2$-transitive action and $M,N \lhd G$ two normal subgroups such that $[M,N] := \langle \{[m,n] \  | \ m \in M, n \in N\} \rangle = \trivgp$. Then either $M = N$ is an abelian normal subgroup, or one of these groups is trivial. 
\end{lemma}
\begin{proof}
Assume that both groups are nontrivial. Since the action is primitive faithful both of them are transitive on $X$. Moreover they act freely. Indeed fix $x \in X$ then $N_{x^m} = m^{-1}N_x m = N_x$, for all $m\in M$, and as $M$ is transitive $N_x$ fixes $X$ pointwise, so it is trivial. Thus $G = B \ltimes N$, where $B = G_x$. We claim that $M \leq N$. Indeed assume $e \ne m \in M$ and let us decompose it according to the above semidirect product $m = bn$. So $b = mn^{-1}$ must commute with $m$. But this means that $b$ fixes two different points $x,x^m \in X$. Since the action is sharply $2$-transitive this means that $b=e$ and $m \in N$. Thus $M \leq N$ and by symmetry the two groups are equal. 
\end{proof} 
\begin{proof}[Proof of Theorem \ref{thm:s2t}]
Let $X \curvearrowleft L$ be the given action, fix any basepoint $x \in X$ and let $A = L_x$ be its stabilizer. The condition that the action on (ordered) pairs of distinct points is free is equivalent  to $A$ being malnormal in H. Since $n \ge 3$ we can find an involution $t \in \SL_n(\R)$, with eigenspaces $W^{\pm}$, such that $\dim(W^{-}) \ge 2$. Indeed, since we are working inside $\SL_n(\R)$, $\dim(W^{-})$ must be even. Let $\Pr = \R^n \arrow W^{-}$ be the projection on $W^{-}$ along $W^{+}$. Applying Proposition \ref{prop:free_prod} to the groups $L$ and $\langle t \rangle$ and then replacing if necessary the given group $L$ by its conjugate obtained in that Proposition we may assume that $G := \langle L, t \rangle \cong L * (\Z/2\Z)$ and that furthermore: (i) $A$ is malnormal in $G$, (ii) all involutions in $G$ are conjugate to $t$ and (iii) if $W^{-}$ is contained in an eigenspace of some $g \in G$ then $g \in \{1,t\}$. Indeed (i),(ii) are guaranteed by the properties of free products and (iii) follows from Proposition \ref{prop:free_prod}.

The conditions (i),(ii),(iii) are exactly these needed in order to apply Corollary \ref{cor:tec} to the groups $A<G$. This yields an embedding of the given action $$(X \curvearrowleft L) \embedding (A \bs G \curvearrowleft G) \embedding (B \bs H \curvearrowleft H).$$
into a sharply $2$-transitive permutation group $H < \SL_n(\R)$ with $\pchar(H) = 2$. This group is not split. Indeed since $\pchar(H) = 2$, if this group were to split the normal abelian subgroup, as the additive group of a near field of characteristic $2$, would be an infinite elementary abelian $2$-group. And this is impossible inside $\SL_n(\R)$. 

Finally we appeal to the main theorem of \cite{GG:AOS} (see also \cite{GG:Primitive}) to deduce that $H$ admits uncountably many nonequivalent primitive permutation representations. According to the main theorem of that paper all we have to verify that $H$ is neither of Affine nor of diagonal type. Being affine is equivalent to the existence of a nontrivial abelian normal subgroup, which we just ruled out.  Being diagonal is equivalent to the existence of two nontrivial commuting normal subgroups which is ruled out by Lemma \ref{lem:diagonal}. This completes the proof of the theorem. 
\end{proof}		

\begin{remark}
The fact that the group we construct contains neither an abelian normal subgroup, nor two commuting normal subgroups can also be seen directly from the construction. This was our original proof. We are grateful to the referee who provided the argument in it's current form, complete with the proof of Lemma \ref{lem:diagonal}, which we find more pleasing. 
\end{remark}

% BibTeX users please use
\bibliographystyle{alpha}
\bibliography{../tex_utils/yair}

\noindent {\sc Yair Glasner.} Department of Mathematics.
Ben-Gurion University of the Negev.
P.O.B. 653,
Be'er Sheva 84105,
Israel.
{\tt yairgl\@@math.bgu.ac.il}\bigskip

\noindent {\sc Dennis D. Gulko.} Department of Mathematics.
Ben-Gurion University of the Negev.
P.O.B. 653,
Be'er Sheva 84105,
Israel.
{\tt gulko\@@math.bgu.ac.il}\bigskip

\end{document}